\documentclass[10pt]{amsart}

\usepackage{ amsmath,amssymb,verbatim,graphicx,amsthm}
\usepackage{color, soul}
\usepackage[top=1in, bottom=1in, left=1in, right=1in]{geometry}
\usepackage{cite}
 \usepackage{url}
 
 \allowdisplaybreaks[1]

\usepackage[sc]{mathpazo}
\usepackage[T1]{fontenc}

\newcommand{\vertiii}[1]{{\left\vert\kern-0.25ex\left\vert\kern-0.25ex\left\vert #1 
    \right\vert\kern-0.25ex\right\vert\kern-0.25ex\right\vert}}
\newcommand{\norm}[1]{\left\lVert#1\right\rVert}

\newcommand{\RR}{\mathbb{R}}

\newcommand{\ZZ}{\mathbb{Z}}

\newcommand{\NN}{\mathbb{N}}

\newcommand{\expec}{\mathbb{E}}
\newcommand{\prob}{\mathbb{P}}

\newcommand{\ind}{\mathbb{1}}

\newcommand{\PP}{\mathbb{P}}
\theoremstyle{plain}

\newtheorem{theorem}{Theorem}[section]
\newtheorem*{theorem*}{Theorem}
\newtheorem{prop}[theorem]{Proposition}
\newtheorem{lemma}[theorem]{Lemma}

\newtheorem{mydef}[theorem]{Definition}

\usepackage{setspace}
\begin{document}

\title{A good universal weight for nonconventional ergodic averages in norm}

\author{Idris Assani}
\address{Department of Mathematics, The University of North Carolina at Chapel Hill, 
Chapel Hill, NC 27599}
\email{assani@math.unc.edu}
\urladdr{http://www.unc.edu/math/Faculty/assani/} 

\author{Ryo Moore}
\address{Department of Mathematics, The University of North Carolina at Chapel Hill, 
Chapel Hill, NC 27599}
\email{ryom@live.unc.edu}
\urladdr{http://ryom.web.unc.edu} 

\begin{abstract}
We will show that the sequence appearing in the double recurrence theorem is a good universal weight for the Furstenberg averages. That is, given a system $(X, \mathcal{F}, \mu, T)$ and bounded functions $f_1, f_2 \in L^\infty(\mu)$, there exists a set of full-measure $X_{f_1, f_2}$ in $X$ that is independent of integers $a$ and $b$ and a positive integer $k$ such that for all $x \in X_{f_1, f_2}$ and for every other measure-preserving system $(Y, \mathcal{G}, \nu, S)$, and each bounded and measurable function $g_1, \ldots, g_k \in L^\infty(\nu)$, the averages 
\[ \frac{1}{N} \sum_{n=1}^N f_1(T^{an}x)f_2(T^{bn}x)g_1 \circ S^n g_2 \circ S^{2n} \cdots g_k \circ S^{kn} \]
converge in $L^2(\nu)$.
\end{abstract}

\maketitle

\section{Introduction}
\subsection{Background}
\subsubsection{Good universal weights}
In some literatures (e.g. \cite[Definitions 3.1-3.3]{AssaniWWET}), the sequence $(a_n)$ is called a \textit{good universal weight for the pointwise ergodic theorem} if for any probability measure preserving system $(Y, \mathcal{G}, \nu, S)$ and any $g \in L^\infty(\nu)$, the averages
\begin{equation}\label{classical} \frac{1}{N} \sum_{n=0}^{N-1} a_n g(S^ny) \end{equation}
converge for $\nu$-a.e. $y \in Y$. Similarly, the sequence $(a_n)$ is called a\textit{ good universal weight for the mean ergodic theorem} if the averages in $(\ref{classical})$ converge in $L^2(\nu)$.

{In this paper, we will extend these classical notions of good universal weights to discuss the case where the sequence $(g \circ S^n)_n$ in $(\ref{classical})$ is replaced by other sequences of bounded and measurable functions $(X_n)_n$.
\begin{mydef}
	We say $(X_n)_n$ is a \textbf{process} if for all nonnegative integers $n \geq 0$, $X_n$ is a bounded and measurable function on some probability measure space $(\Omega, \mathcal{S}, \PP)$.
\end{mydef}
For instance, a sequence of bounded and measurable functions $(X_n)_n = (g \circ S^n)_n$ for any $g \in L^\infty(\nu)$ on any probability measure-preserving system $(Y, \mathcal{G}, \nu, S)$ is a process. Another process $(X_n)_n$ of our interest is a product of multiple functions each iterated by different powers of a measure-preserving transformation, such as
\[ X_n(y) = g_1 (S^ny) g_2  (S^{2n}y) \cdots  g_k (S^{kn}y), \]
for any positive integer $k$, where $g_1, g_2, \ldots, g_k \in L^\infty(\nu)$ on any measure-preserving system $(Y, \mathcal{G}, \nu, S)$.}
\begin{mydef}\label{guweight}
	We denote by
	\[ M_1 = \left\{ (a_n) : \sup_N \frac{1}{N} \sum_{n=1}^N |a_n| < \infty \right\}. \]
	\begin{itemize} 
		\item We say a sequence $(a_n) \in M_1$ is a \textbf{good universal weight for $(X_n)_n$ (a.e.) pointwise} if for any probability measure space $(\Omega, \mathcal{S}, \prob)$ for which the process $(X_n)_n$ is defined, the averages
		\[\frac{1}{N} \sum_{n=1}^N a_n X_n(\omega) \]
		converge for $\prob$-a.e. $\omega \in \Omega$.
		\item We say a sequence $(a_n) \in M_1$ is a \textbf{good universal weight for $(X_n)_n$ in norm} if for any probability measure space $(\Omega, \mathcal{S}, \prob)$ for which the process $(X_n)_n$ is defined, the averages
		\[\frac{1}{N} \sum_{n=1}^N a_n X_n(\omega) \]
		converge in $L^2(\prob)$
	\end{itemize}
\end{mydef}
{For example, if $(a_n)_n$ is a good universal weight for the process $(X_n)_n = (g \circ S^n)_n$ pointwise (resp. in norm) for any $g \in L^\infty(\nu)$, where $(Y, \mathcal{G}, \nu, S)$ is any probability measure-preserving system, then $(a_n)_n$ is a good universal weight for the pointwise ergodic theorem (resp. the mean ergodic theorem) in the classical sense.}
\subsubsection{History of the return times theorem}
{The studies of the return times theorem have shown that we can randomly generate good universal weights. The basic principle of the return times theorem that has been initially studied by A. Brunel in his Ph.D. thesis in 1966 \cite{BrunelThesis} is as follows: Given a process $X_n(\omega) $ converging in average (in norm or pointwise) and the characteristic function of a measurable set with positive measure, $\mathbb{1}_A$,  do we still have the convergence of the averages along the subsequence given by the return times of $T^nx$ to the set $A$?  In other words, is the sequence $(\mathbb{1}_A (T^nx))_n$ a good universal weight (in norm or pointwise)  for the averages of $\mathbb{1}_A(T^nx) X_n(\omega)$? In 1969, A. Brunel and M. Keane answered this question positively for a particular class of dynamical systems for both pointwise and norm convergence \cite{BrunelKeane69}. Krengel's book highlights some of the generalization of their work \cite{Krengel}.}

One of the important results in ergodic theory is the proof of return times theorem by J. Bourgain \cite{BourgainRT}, which was later simplified by J. Bourgain, H. Furstenberg, Y. Katznelson, and D. Ornstein (a.k.a. the "BFKO" argument) \cite{BFKO}. This result strengthens Birkhoff's pointwise ergodic theorem and generalizes the above-mentioned results on return times.

\begin{theorem}[Bourgain's Return Times Theorem]
	Let $(X, \mathcal{F}, \mu, T)$ be a probability measure-preserving system and $f \in L^\infty(\mu)$. Then there exists a set $X_f \subset X$ of full measure such that for any other probability measure-preserving system $(Y, \mathcal{G}, \nu, S)$ and any $g \in L^\infty(\nu)$,
	\[ \frac{1}{N} \sum_{n=1}^N f(T^nx) g(S^ny) \]
	converges $\nu$-almost everywhere for all $x \in X_f$. 
\end{theorem}
While the set of full-measure $X_f$ depends on the function $f$ and the transformation $T$, it is independent of every other ergodic system. {In terms of Definition $\ref{guweight}$, Bourgain has shown that for $\mu$-a.e. $x \in X$, the sequence $a_n = f(T^nx)$ is a good universal weight for $(X_n)_n = (g \circ S^n)_n$ pointwise, where $g \in L^\infty(\nu)$ for any measure-preserving system $(Y, \mathcal{G}, \nu, S)$.}

\subsubsection{Extensions of the return times theorem}
{Much of the background, historical development, and current status of the return times can be found in the survey paper prepared by the first author and K. Presser \cite{AssaniPresserSurvey}. Here, we will focus on discussing some of the developments on the return times theorem regarding mixing of multiple recurrence and multi-term return times problems. Some new results that appeared since the emergence of the survey paper are mentioned as well.}

Since the result of Bourgain emerged, the return times theorem has been extended in multiple direction. One way is to find a new universal weight in which the return-times averages converge. For instance, the first author shows in Proposition 5.3 of \cite{AssaniWWET} that if $(X, \mathcal{F}, \mu, T)$ is a weakly-mixing, standard uniquely ergodic system with Lebesgue spectrum, and $f \in \mathcal{C}(X)$, then $(f(T^nx))$ is a good universal weight for the pointwise ergodic theorem \textit{for all} $x \in X$. Recently, P. Zorin-Kranich announced the extension of Bourgain's return times theorem by showing that the double recurrence sequence is a good universal weight for the pointwise ergodic theorem for $\mu$-a.e. $x \in X$ \cite{ZK_DRT}.

The return times theorem has also been extended to averages with more than two terms. One example of such is the multiterm return times theorem that was obtained by D. Rudolph in 1998 \cite{RudolphMTRT}, which answers one of the questions raised by the first author in 1991. Rudolph's proof utilized the method of joinings and fully generic sequences, while the method of factor decomposition was absent, which was one of the key tools in the BFKO argument of the return times theorem. Later, the first author and K. Presser identified characteristic factors for the multiterm return times theorem \cite{AssaniPresser03, AssaniPresser}. Furthermore, P. Zorin-Kranich provided a different proof of the multiterm return times theorem based on these factor structures, and showed that multiterm return times averages can be extended to Wiener-Wintner type averages with nilsequences \cite{ZK_MRT}. Also, T. Eisner \cite{EisnerLinear} showed the convergence of Wiener-Wintner type averages for multiterm return times theorem with linear sequences.

In another direction, the return times theorem has been extended by mixing weights from the a.e. multiple recurrence and the multiterm return times theorem. This idea was introduced by the first author in 1998, in which he proved the following:
\begin{theorem}[{\cite[Theorem 3]{AssaniMRT_Weakly}}]\label{MRT_Weakly}
	Let $(X, \mathcal{F}, \mu, T)$ be a weakly mixing dynamical system such that for all positive integers $H$, for all $f_1, f_2, \ldots f_H \in L^\infty(\mu)$, for all $(b_1, b_2, \ldots, b_H)\in \ZZ^H$ where $b_i$ distinct and not equal to zero, the sequence
	\[ \frac{1}{N} \sum_{n=1}^N \left(\prod_{i=1}^H f_i(T^{b_in}x)\right) \text{ converges a.e. to } \prod_{i=1}^H \int f_i d\mu\, . \]
	Then there exists a set of full measure $X'$ for any other weakly mixing system $(Y_1, \mathcal{G}_1, S_1, \nu_1)$ and any $g_1 \in L^\infty(\nu_1)$, there exists a set of full measure $Y_{g_1}$ in $Y_1$ such that if $y_1 \in Y_{g_1}$, then . . . for any other weakly mixing system $(Y_{k-1}, \mathcal{G}_{k-1}, S_{k-1}, \nu_{k-1})$ and any $g_{k-1} \in L^\infty(\nu_{k-1})$ there exists a set of full measure $Y_{g_{k-1}}$ in $Y_{k-1}$ such that if $y_{k-1} \in Y_{g_{k-1}}$, then for any other weakly mixing system $(Y_k, \mathcal{G}_k, S_k, \nu_k)$, the sequence
	\[ {\xi}_n(x, y_1, y_2, \ldots, y_{k}) = \left(\prod_{i=1}^H f_i(T^{b_in}x)\right) \left(\prod_{j=1}^{k} g_j(S^n_j y_j)\right) \]
	is a good universal weight for the pointwise ergodic theorem for $\nu_k$-a.e. $y_k \in Y_k$.
\end{theorem}

{For instance, if $(X, \mathcal{F}, \mu, T)$ is a weakly mixing system for which the restriction of $T$ to its Pinsker algebra has singular spectrum, then the hypothesis of the theorem above holds. This result was proven by the first author in 1998 \cite{Assani_MultRecWeaklyMix}.}

{In terms of Definition $\ref{guweight}$, Theorem $\ref{MRT_Weakly}$ says that for $k = 1$, there exists a set of full-measure $Y_{g_1} \subset Y_1$ such that for all $y_1 \in Y_1$, the sequence $(\prod_{i=1}^H f_i(T^{b_in}x))_n$ is a good universal weight for $\mu$-a.e. $x \in X$ for the process $X_n(z) = X_n[y_1, g_1, S_1](z) = g_1(S_1^ny_1)h(R^nz)$ pointwise, for any measure-preserving system $(Z, \mathcal{Z}, \eta, R)$ and a function $h \in L^\infty(\eta)$.}

{In 2009, B. Host and B. Kra showed in \cite{HostKraUniformity} that given an ergodic system $(X, \mathcal{F}, \mu, T)$ and $f \in L^\infty(\mu)$, the sequence $(f(T^nx))$ is a good universal weight for $\mu$-a.e. $x \in X$ for the convergence in $L^2$-norm of the Furstenberg averages, i.e. they have shown that there exists a set of full-measure $X' \subset X$ such that for any $x \in X'$ and any other measure-preserving system $(Y, \mathcal{G}, \nu, S)$ with functions $g_1, \ldots, g_k \in L^\infty(\nu)$, the averages
\begin{equation}\label{HostKra}
\frac{1}{N} \sum_{n=1}^N f(T^nx) \prod_{i=1}^k g_i \circ S^{in},
\end{equation}
converge in $L^2(\nu)$. In particular, if $f = \ind_A$ for some measurable set $A \in \mathcal{F}$, then they have shown that the averages of the sequence $(\prod_{i=1}^k g_i (S^{in}y))_n$ along the subsequence of the return times of $T^nx$ to the set $A$ converge in $L^2(\nu)$-norm. {In the language of Definition $\ref{guweight}$, for $\mu$-a.e. $x \in X$, the sequence $(f(T^nx))$ is a \textbf{good universal weight for $(X_n)_n$ in norm}, where $(X_n)_n$ is a process of the form 
\begin{equation}\label{theProcesses} (X_n)_n = \left( \prod_{i=1}^k g_i \circ S^{in}\right)_n \text{ for any } g_1, \ldots, g_k \in L^\infty(\nu) \text{ on any m.p.s. } (Y, \mathcal{G}, \nu, S), \text{ for any }k \geq 1.  \end{equation}
}This result extends their earlier work in \cite{HostKraNEA}, where they proved the result for $f = \ind_X$. To show this result, they used the machinery of nilsequences {(see \cite{BHK_nilseq,HostKraUniformity} for more background on nilsequences)}; they showed that if a bounded sequence $(a_n)_n \in \ell^\infty$ has a property that the Cesaro averages of $a_n b_n$ converge for any $k$-step nilsequence $(b_n)_n$, then $(a_n)_n$ is a good universal weight for $k$-term multiple recurrent averages in the $L^2$-norm. Then the convergence of the averages in $(\ref{HostKra})$ follows from the fact that there exists a set of full-measure $X'$ so that for any $x \in X'$ and any nilsequence $(b_n)_n$, the Cesaro averages of $f(T^nx)b_n$ converge; this is referred to as the generalized Wiener-Wintner theorem. Later, in the work of T. Eisner and P. Zorin-Kranich, the generalized Wiener-Wintner theorem was extended to any measure-preserving system (with not necessarily ergodic transformation) with uniform counterpart, and used this to extend the result to a case with polynomial actions \cite{EZK2013}. 

\subsection{The main theorem}
In this paper, we will prove the following:
\begin{theorem}[The main result]\label{mainResult}
Let $(X, \mathcal{F}, \mu, T)$ be a probability measure-preserving system, with functions $f_1, f_2 \in L^\infty(\mu)$. Then for $\mu$-a.e. $x \in X$, the sequence $u_n = (f_1(T^{an}x)f_2(T^{bn}x))_n$ is a good universal weight for a $k$-term Furstenberg averages in norm for any positive integer $k$. More precisely, there exists a set of full-measure $X_{f_1, f_2} \subset X$ such that for any $x \in X_{f_1, f_2}$, $a, b \in \ZZ$ and any positive integer $k \geq 1$, and any other probability measure-preserving system $(Y, \mathcal{G}, \nu, S)$ with $g_1, \ldots, g_k \in L^\infty(\nu)$, the averages
	\[ \frac{1}{N} \sum_{n=1}^N f_1(T^{an}x)f_2(T^{bn}x) \prod_{i=1}^k g_i \circ S^{in} \]
converge in $L^2(\nu)$.
\end{theorem}
In particular, if $f_1 = \ind_A$ and $f_2 = \ind_B$ for some measurable sets $A, B \in \mathcal{F}$ with positive measures, then we see that the averages of the sequence $(\prod_{i=1}^k g_i (S^{in}y))_n$ along the subsequence of the return times of $T^{an}x$ to the set $A$ and $T^{bn}x$ to the set $B$ converge in $L^2(\nu)$-norm. This theorem mixes the weights from the a.e. double recurrent convergence result and the norm convergence of the multiple recurrent theorem. {In terms of Definition $\ref{guweight}$, we show that for $\mu$-a.e. $x \in X$, the sequence $(f_1(T^{an}x)f_2(T^{bn}x))_n$ is a good universal weight for the process $(X_n)_n$ of the form in $(\ref{theProcesses})$ in norm.} Note that this theorem generalizes the result obtained by B. Host and B. Kra, since if $a = 1$ and $f_2 = \ind_X$, then the averages in the theorem become the averages seen in $(\ref{HostKra})$.

The Cesaro averages of the sequence $(f_1(T^{an}x)f_2(T^{bn}x))_n$ is known to converge for $\mu$-a.e. $x \in X$ by Bourgain's double recurrence theorem \cite{BourgainDR}. It was recently extended to a Wiener-Wintner result \cite{WWDR}, {and further to a polynomial Wiener-Wintner result \cite{WWDR_poly}}. Note that the case $k = 1$ of the main result follows immediately from this Wiener-Wintner result. In fact, we will show that this is the key step required to establish the "base case" of our inductive argument in the proof.

In the proof, we will assume that the systems $(X, \mathcal{F}, \mu, T)$ and $(Y, \mathcal{G}, \nu, S)$ are ergodic, since we can apply the ergodic decomposition to show that the result holds for general measure-preserving systems. To prove the main result for $k \geq 2$, we will first decompose the functions $f_1$ and $f_2$ into an appropriate characteristic factor of $(X, \mathcal{F}, \mu, T)$, and treat the cases when either $f_1$ or $f_2$ belongs to the orthogonal complement of this factor, or the case when both of them belong to the factor. For the first case, we will prove it by induction on $k$. We will show that the case $k = 2$ follows from the fact that the theorem holds for the case $k = 1$; to do so, we will show that the $L^2(\nu)$-norm limit of the averages can be controlled by the limit of the double recurrence Wiener-Wintner averages. We will also show that the case $k = 3$ follows from the case $k = 2$ to demonstrate the inductive step necessary to prove this for any $k \in \NN$. For the second case, we will decompose the functions $g_1, \ldots, g_k$ into an appropriate characteristic factor, and treat the sub-cases when either one of $g_1, \ldots, g_k$ belongs to the orthogonal complement of this factor, and when all of them belong to the factor separately. For the first sub-case, we will control the norm limit of the averages with a seminorm that characterizes this factor, and uses this to show that the norm averages converge to $0$. For the second sub-case, we will use the structure of nilmanifolds and Leibman's convergence result \cite{Leibman} to prove the claim.

The factors we use are the Host-Kra-Ziegler factors \cite{HostKraNEA, Ziegler}. Throughout this paper, we denote $\mathcal{Z}_k(T)$ to be the $k$-th Host-Kra-Ziegler factor of $(X, T)$, which is characterized by the $k+1$-th Gowers-Host-Kra seminorm $\vertiii{ \cdot }_{k+1}$ \cite{Gowers, HostKraNEA}. Using the language of these factors, Theorem $\ref{mainResult}$ can be shown by proving the following:
\begin{theorem}\label{mainThm}
	Let $(X, \mathcal{F}, \mu, T)$ be an ergodic system, and $f_1, f_2 \in L^\infty(\mu)$ such that $\norm{f_i}_{L^\infty(\nu)} \leq 1$ for both $i = 1, 2$. Fix a positive integer $k \geq 1$. Then the following statements are true.
	\begin{enumerate}
		\item[(a)] Suppose either $f_1, f_2 \in \mathcal{Z}_{k+1}(T)^\perp$. Then there exists a set of full-measure $\tilde{X}_k \subset X$ such that for any $x \in \tilde{X}$, any other measure-preserving system $(Y, \mathcal{G}, \nu, S)$, and functions $g_1, g_2, \ldots, g_k \in L^\infty(\nu)$ where $\norm{g_j}_{L^\infty(\nu)} \leq 1$ for each $1 \leq j \leq k$, the averages
		\begin{equation}\label{mainAvg}
		\frac{1}{N} \sum_{n=1}^N f_1(T^{an}x)f_2(T^{bn}x) \prod_{j=1}^k g_j \circ S^{jn}
		\end{equation}
		converge to $0$ in $L^2(\nu)$.
		\item[(b)] For any $f_1, f_2 \in L^\infty(\mu)$, there exists a set of full-measure $\hat{X}_k$ such that for any $x \in \hat{X}$ and any other ergodic system $(Y, \mathcal{G}, \nu, S)$, and functions $g_1, \ldots, g_k \in L^\infty(\nu)$ with one of them belonging to $\mathcal{Z}_k(S)^\perp$, the averages in $(\ref{mainAvg})$ converge to $0$ in $L^2(\nu)$.
		\item[(c)] Suppose both $f_1, f_2 \in \mathcal{Z}_{k+1}(T)$. Then there exists a set of full-measure $X_k' \in X$ such that for any $x \in X'$, and for any other ergodic system $(Y, \mathcal{G}, \nu, S)$ and functions $g_1, \ldots, g_k \in L^\infty(\nu) \cap \mathcal{Z}_{k}(S)$, the averages in $(\ref{mainAvg})$ converge in $L^2(\nu)$.
	\end{enumerate}
\end{theorem}
\begin{proof}[Proof that Theorem $\ref{mainThm}$ implies Theorem $\ref{mainResult}$.]{Fix a positive integer $k \geq 1$. Let $f'_i = f_i - \expec(f_i|\mathcal{Z}_{k+1})$ for $i = 1, 2$. We rewrite the averages in $(\ref{mainAvg})$ as follows:
\begin{align}\label{decomposition}
&\frac{1}{N} \sum_{n=1}^N f_1(T^{an}x)f_2(T^{bn}x) \prod_{j=1}^k g_j \circ S^{jn} \\
&= \frac{1}{N} \sum_{n=1}^N \expec(f_1|\mathcal{Z}_{k+1})(T^{an}x)\expec(f_2|\mathcal{Z}_{k+1})(T^{bn})\prod_{j=1}^k g_j \circ S^{jn} \nonumber\\ &+\frac{1}{N}\sum_{n=0}^{N-1}\expec(f_1|\mathcal{Z}_{k+1})(T^{an}x)f_2'(T^{bn}x) \prod_{j=1}^k g_j \circ S^{jn} \nonumber\\
&+\frac{1}{N} \sum_{n=1}^N f_1'(T^{an}x)\expec(f_2|\mathcal{Z}_{k+1})(T^{bn}x) \prod_{j=1}^k g_j \circ S^{jn} \nonumber \\
&+\frac{1}{N} \sum_{n=1}^N f_1'(T^{an}x)f_2'(T^{bn}x) \prod_{j=1}^k g_j \circ S^{jn} \nonumber.
\end{align}
We know that, by Theorem $\ref{mainThm}$(a), there exists a universal set of full-measure $\tilde{X}_k$ such that for all $x \in \tilde{X}_k$, the last three averages of the right hand side of $(\ref{decomposition})$ converge to $0$ in $L^2(\nu)$. And by Theorem $\ref{mainThm}$(b-c), the first averages also converge in $L^2(\nu)$ for all $x \in \hat{X}_k \cap X'_k$. So if we set
\[ X_{f_1, f_2, k} = \tilde{X}_k \cap \hat{X}_k \cap X'_k, \]
then $X_{f_1, f_2, k}$ is a set full-measure that only depends on $f_1$, $f_2$, the transformation $T$, and the positive integer $k$, since it is a finite intersection of the sets of full-measure, each only depending on the functions $f_1$, $f_2$, and the transformation $T$. Thus, for any $x \in X_{f_1, f_2, k}$, $a, b \in \ZZ$, and any other ergodic system $(Y, \mathcal{G}, \nu, S)$ with functions $g_1, \ldots, g_k \in L^\infty(\nu)$, the averages in $(\ref{mainAvg})$ converge in $L^2(\nu)$. This implies that the set
\[ X_{f_1, f_2} = \bigcap_{k=1}^\infty X_{f_1, f_2, k} \]
is a set of full-measure that only depends on the functions $f_1$, $f_2$, and the transformation $T$, and this is indeed the desired universal set for Theorem $\ref{mainResult}$.}
\end{proof}
\subsection{Organization of the paper}
In \S$\ref{outside}$, we will prove (a) of Theorem $\ref{mainThm}$, which treats the case where either $f_1$ or $f_2$ belongs to the orthogonal complement of the appropriate factor of $(X, T)$. The case where $f_1$ and $f_2$ both belong to the appropriate factor is discussed in \S$\ref{inside}$, where we first look at the case where either one of the functions $g_1, \ldots, g_k$ belongs to the orthogonal complement of the appropriate factor of $(Y, S)$ (which corresponds to (b) of Theorem $\ref{mainThm}$), and the case all of them belong to the appropriate factor (which corresponds to (c) of Theorem $\ref{mainThm}$). 
{\subsection{Acknowledgment} We thank the anonymous referee for his/her comments.}
\section{The case where either $f_1$ or $f_2$ belongs to $\mathcal{Z}_{k+1}(T)^\perp$ (Proof of (a) of Theorem $\ref{mainThm}$)}\label{outside}
The idea of the proof is as follows: We will first prove the statement for the case $k = 2$. We first identify the set of full-measure for which the averages in $(\ref{mainAvg})$ converges to $0$; the fact that this is indeed a set of full-measure can be shown by using Fatou's lemma and the following inequality obtained in \cite{WWDR}:
\begin{equation}\label{estimate}
\int \limsup_{N \to \infty} \sup_{t \in \RR} \left|\frac{1}{N} \sum_{n=0}^{N-1} f_1(T^{an}x)f_2(T^{bn}x) e^{2\pi i n t} \right|^2 d\mu(x) \lesssim_{a, b} \min_{i = 1, 2}\vertiii{f_i}_3^2.
\end{equation}
 The key observation of the proof is the fact that $S$ is a measure-preserving transformation allows us to bound the $L^2(\nu)$-norm of the averages by the double recurrence Wiener-Wintner averages; to do so, we apply van der Corput's lemma \cite{KuipersNiederreiter}, H\"{o}lder's inequality, and the spectral theorem. This allows us to show that the averages in $(\ref{mainAvg})$ indeed converge to $0$ when $k = 2$ for this set of full-measure.

Then we will proceed for the case $k = 3$ to demonstrate that the claim can be proven inductively for the case $k > 2$. Again we start by identifying the set of full-measure. To show that the averages converge to $0$ on this set, we rely on the result obtained for the case $k = 2$.

Before we prove this part of the theorem, we will prove this for the case where $k = 2, 3$ to demonstrate the inductive step for simple cases. For the case $k=2$, we would like to show that
\begin{equation}\label{unifTwoTerms}
\limsup_{N \to \infty} \norm{\frac{1}{N} \sum_{n=1}^N f_1(T^{an}x)f_2(T^{bn}x) g_1 \circ S^n  g_2 \circ S^{2n}}_{L^2(\nu)}^2 = 0.
\end{equation}
Consider a set
\begin{equation}
\tilde{X}_2 = \left\{ x \in X : \liminf_{H \to \infty} \left(\frac{1}{H}\sum_{h=1}^H \limsup_{N \to \infty} \sup_{t \in \RR}\left| \frac{1}{N} \sum_{n=1}^{N-h} f_1 \cdot f_1 \circ T^{ah}(T^{an}x) f_2 \cdot f_2 \circ T^{bh}(T^{bn}x)e^{2\pi i n t}\right|^2\right)^{{1/2}} = 0  \right\}
\end{equation}
First we show that $\tilde{X}_2$ is a set of full-measure. To do so, we apply Fatou's lemma and the inequality $(\ref{estimate})$ to obtain
\begin{align*}
&\int  \liminf_{H \to \infty}  \left(\frac{1}{H}\sum_{h=1}^H \limsup_{N \to \infty} \sup_{t \in \RR}\left| \frac{1}{N} \sum_{n=1}^{N-h} f_1 \cdot f_1 \circ T^{ah}(T^{an}x) f_2 \cdot f_2 \circ T^{bh}(T^{bn}x)e^{2\pi i n t}\right|^2 \right)^{{1/2}}d\mu \\
&\leq \liminf_{H \to \infty} \left(\frac{1}{H}\sum_{h=1}^H \int  \limsup_{N \to \infty} \sup_{t \in \RR}\left| \frac{1}{N} \sum_{n=1}^{N-h} f_1 \cdot f_1 \circ T^{ah}(T^{an}x) f_2 \cdot f_2 \circ T^{bh}(T^{bn}x)e^{2\pi i n t}\right|^2 d\mu \right)^{{1/2}} \\
&\lesssim_{a, b} \min_{i=1, 2}  \liminf_{H \to \infty} \frac{1}{H}\sum_{h=1}^H \vertiii{f_i \cdot f_i \circ T^h}_3  \leq \min_{i=1, 2} \left( \liminf_{H \to \infty} \frac{1}{H}\sum_{h=1}^H \vertiii{f_i \cdot f_i \circ T^h}_3^8 \right)^{{1/8}} \\
&= \min_{i=1, 2} \vertiii{f_i}_4^2.
\end{align*}
Since either $f_1$ or $f_2$ belongs to $\mathcal{Z}_3(T)^\perp$, either $\vertiii{f_1}_4$ or $\vertiii{f_2}_4$ equals zero. This shows that $\mu(\tilde{X}_2) = 1$.

Now we claim $(\ref{unifTwoTerms})$ holds for all $x \in \tilde{X}_2$. In fact, we show that for any $1 \leq H < N$, we have
\begin{align}\label{dynamicEstimateTwoTerms}
&\norm{ \frac{1}{N} \sum_{n=1}^N f_1(T^{an}x)f_2(T^{bn}x)g_1 \circ S^n g_2 \circ S^{2n}}_{L^2(\nu)}^2 \nonumber\\
&\lesssim_{a, b} \frac{1}{H} + \left( \frac{1}{H} \sum_{h=1}^H \sup_{t \in \RR} \left| \frac{1}{N} \sum_{n=1}^{N-h} f_1 \cdot f_1 \circ T^{ah}(T^{an}x)f_2 \cdot f_2 \circ T^{bh}(T^{bn}x)e^{2\pi i n t} \right|^2 \right)^{1/2}
\end{align}
To do so, we proceed with van der Corput's lemma; using the fact that $S$ is a measure preserving transformation, we obtain
\begin{align*}
&\norm{ \frac{1}{N} \sum_{n=1}^N f_1(T^{an}x)f_2(T^{bn}x)   g_1 \circ S^n  g_2 \circ S^{2n}}_{L^2(\nu)}^2 \\
&\leq \frac{2}{H} + \frac{4}{H} \sum_{h=1}^H\left| \int \frac{1}{N} \sum_{n=1}^{N-h} f_1 \cdot f_1 \circ T^{ah}(T^{an}x)f_2 \cdot f_2 \circ T^{bh}(T^{bn}x)(g_1 \cdot  g_1 \circ S^h) \circ S^n (g_2 \cdot  g_2 \circ S^{2h}) \circ S^{2n} d\nu \right| \\
&= \frac{2}{H} + \frac{4}{H} \sum_{h=1}^H  
\left| \int \frac{1}{N} \sum_{n=1}^{N-h} f_1 \cdot f_1 \circ T^{ah}(T^{an}x)f_2 \cdot f_2 \circ T^{bh}(T^{bn}x) (g_1 \cdot  g_1 \circ S^h)(g_2 \cdot g_2 \circ  S^{2h}) \circ S^{n}  d\nu \right| \\
&\leq \frac{2}{H} + \frac{4}{H} \sum_{h=1}^H \norm{ 
 \frac{1}{N} \sum_{n=1}^N f_1 \cdot f_1 \circ T^{ah}(T^{an}x)f_2 \cdot f_2 \circ T^{bh}(T^{bn}x) (g_2 \cdot  g_2 \circ S^{2h}) \circ S^n }_{L^2(\nu)} \text{ (by H\"{o}lder's inequality)}\\
&\leq \frac{2}{H} + \left(\frac{16}{H} \sum_{h=1}^H \norm{ 
	\frac{1}{N} \sum_{n=1}^N f_1 \cdot f_1 \circ T^{ah}(T^{an}x)f_2 \cdot f_2 \circ T^{bh}(T^{bn}x) (g_2 \cdot  g_2 \circ S^{2h}) \circ S^n }_{L^2(\nu)}^2 \right)^{1/2} \, ,
\end{align*}
where the last inequality follows from the Cauchy-Schwarz inequality. We apply the spectral theorem to the square of the $L^2(\nu)$-norm in the last line to obtain
\begin{align*}
&\norm{ 
	\frac{1}{N} \sum_{n=1}^N f_1 \cdot f_1 \circ T^{ah}(T^{an}x)f_2 \cdot f_2 \circ T^{bh}(T^{bn}x) (g_2 \cdot g_2 \circ  S^{2h}) }_{L^2(\nu)}^2 \\
&= \int \left| \frac{1}{N} \sum_{n=1}^N f_1 \cdot f_1 \circ T^{ah}(T^{an}x)f_2 \cdot f_2 \circ T^{bh}(T^{bn}x) e^{2\pi i n t} \right|^2 d\sigma_{g_2 \cdot  g_2 \circ S^{2h}}(t)\\
&\leq \sup_{t \in \RR} \left| \frac{1}{N} \sum_{n=1}^N f_1 \cdot f_1 \circ T^{ah}(T^{an}x)f_2 \cdot f_2 \circ T^{bh}(T^{bn}x) e^{2\pi i n t} \right|^2\, ,
\end{align*}
which tells us that $(\ref{dynamicEstimateTwoTerms})$ holds. Thus, if $x \in \tilde{X}^2$, and we let $N \to \infty$ (and consequently $H \to \infty$) in $(\ref{dynamicEstimateTwoTerms})$, we obtain
\begin{align*}
&\limsup_{N \to \infty} \norm{ \frac{1}{N} \sum_{n=1}^N f_1(T^{an}x)f_2(T^{bn}x)g_1 \circ S^n g_2 \circ S^{2n}}_{L^2(\nu)}^2 \nonumber\\
&\lesssim_{a, b} \left( \liminf_{H \to \infty} \frac{1}{H} \sum_{h=1}^H \limsup_{N \to \infty} \sup_{t \in \RR} \left| \frac{1}{N} \sum_{n=1}^{N-h} f_1 \cdot f_1 \circ T^{ah}(T^{an}x)f_2 \cdot f_2 \circ T^{bh}(T^{bn}x)e^{2\pi i n t} \right|^2 \right)^{1/2} = 0.
\end{align*}
This proves the case for $k = 2$.
Now we show that the holds for the case $k = 3$ using the fact that the convergence to $0$ holds for $k = 2$. We let $F_{1, h_1} = f_1 \cdot f_1 \circ T^{ah_1}$ and $F_{2, h_1} = f_2 \cdot f_2 \circ T^{bh_1}$. Then we set
\begin{align*} 
\tilde{X}_3 &= \left\{ x \in X: \liminf_{H_1 \to \infty} \left(\frac{1}{H_1}\sum_{h_1 = 1}^{H_1} \liminf_{H_2 \to \infty}  \frac{1}{H_2} \sum_{h_2 = 1}^{H_2} \right. \right. \\ 
&\left. \left. \limsup_{N \to \infty} \sup_{t \in \RR} \left| \frac{1}{N} \sum_{n=1}^{N} F_{1, h_1} \cdot F_{1, h_1} \circ T^{ah_2}(T^{an}x)F_{2, h_1} \cdot F_{2, h_1} \circ T^{bh_2}(T^{bn}x)e^{2\pi i n t} \right|^2 \right)^{{1/4}} = 0 \right\}.\end{align*}
We first show that $\tilde{X}_3$ is a set of full-measure. To see that, we apply Fatou's lemma twice to interchange the integral and the $\liminf$'s, H\"{o}lder's inequality, the inequality $(\ref{estimate})$, and the Cauchy-Schwarz inequality multiple times to obtain
\begin{align*}
& \int \liminf_{H_1 \to \infty} \left(\frac{1}{H_1}\sum_{h_1 = 1}^{H_1}  \liminf_{H_2 \to \infty} \frac{1}{H_2}\sum_{h_2 = 1}^{H_2} \right.  \\
&  \left. \limsup_{N \to \infty} \sup_{t \in \RR} \left| \frac{1}{N} \sum_{n=1}^{N} F_{1, h_1} \cdot F_{1, h_1} \circ T^{ah_2}(T^{an}x)F_{2, h_1} \cdot F_{2, h_1} \circ T^{bh_2}(T^{bn}x)e^{2\pi i n t} \right|^2 d\mu(x) \right)^{{1/4}} \\
&\leq \liminf_{H_1 \to \infty} \left( \frac{1}{H_1}\sum_{h_1 = 1}^{H_1} \liminf_{H_2 \to \infty}  \frac{1}{H_2}\sum_{h_2 = 1}^{H_2} \right. \\
& \left. \int \limsup_{N \to \infty} \sup_{t \in \RR} \left| \frac{1}{N} \sum_{n=1}^{N} F_{1, h_1} \cdot F_{1, h_1} \circ T^{ah_2}(T^{an}x)F_{2, h_1} \cdot F_{2, h_1} \circ T^{bh_2}(T^{bn}x)e^{2\pi i n t} \right|^2 d\mu(x) \right)^{{1/4}}\\
&\lesssim_{a_1, a_2} \liminf_{H_1 \to \infty} \left(\frac{1}{H_1}\sum_{h_1 = 1}^{H_1}  \liminf_{H_2 \to \infty} \frac{1}{H_2}\sum_{h_2 = 1}^{H_2} \min_{i=1, 2}\vertiii{F_{i, h_1} \cdot F_{i, h_1} \circ T^{a_ih_2}}_3^2\right)^{{1/4}} \; \mbox{ (where $a_1 = a$, $a_2 = b$)} \\
&\leq \liminf_{H_1 \to \infty}\left( \frac{1}{H_1}\sum_{h_1 = 1}^{H_1}  \left(\liminf_{H_2 \to \infty} \frac{1}{H_2}\sum_{h_2 = 1}^{H_2} \min_{i=1, 2}\vertiii{F_{i, h_1} \cdot F_{i, h_1} \circ T^{a_ih_2}}_3^8 \right)^{1/4} \right)^{{1/4}}\\
&\lesssim_{a_1, a_2} \liminf_{H_1 \to \infty} \left(\frac{1}{H_1} \sum_{h_1=1}^{H_1} \min_{i=1, 2} \vertiii{f_i \cdot f_i \circ T^{a_i h_1}}_4^4\right)^{{1/4}} \\
&\leq \liminf_{H_1 \to \infty} \left(\frac{1}{H_1} \sum_{h_1=1}^{H_1} \min_{i=1, 2} \vertiii{f_i \cdot f_i \circ T^{a_i h_1}}_4^{16} \right)^{{1/16}} \lesssim_{a_1, a_2} \min_{i=1, 2}\vertiii{f_i}_5^{{2}},
\end{align*}
and since either $f_1$ or $f_2$ belongs to $\mathcal{Z}_4(T)^\perp$, either $\vertiii{f_1}_5$ or $\vertiii{f_2}_5$ equals zero. Hence, we know that $\tilde{X}_3$ is a set of full-measure.

Now we will show that the averages converge to $0$ when $x \in \tilde{X}_3$. To do so, we wish to show that
\begin{align*}
&\norm{ \frac{1}{N} \sum_{n=1}^N f_1(T^{an}x)f_2(T^{bn}x)\prod_{j=1}^3 g_j \circ S^{jn}}_{L^2(\nu)}^2 \\
&\lesssim_{a, b} \frac{1}{H_1} + \left(\frac{1}{H_1} \sum_{h_1=0}^{H_1-1} \left( \frac{2}{H_2} + \right. \right. \\
&\left. \left. \left(\frac{16}{H_2}\sum_{h_2=0}^{H_2-1} \sup_{t \in \RR} \left| \frac{1}{N} \sum_{n=1}^{N} F_{1, h_1} \cdot F_{1, h_1} \circ T^{ah_2}(T^{an}x)F_{2, h_1} \cdot F_{2, h_1} \circ T^{bh_2}(T^{bn}x)e^{2\pi i n t} \right|^2\right)\right) \right)^{1/4}
\end{align*}
Indeed, we apply van der Corput's lemma and the Cauchy-Schwarz inequality to show that
\begin{align*}
&\norm{\frac{1}{N} \sum_{n=1}^N f_1(T^{an}x)f_2(T^{bn}x)\prod_{j=1}^3 g_j \circ S^{jn} }_{L^2(\nu)}^2 \\
&\leq \frac{2}{H_1} + \frac{4}{H_1} \sum_{h=1}^H \left| \int \frac{1}{N} \sum_{n=1}^{N-h_1} f_1 \cdot f_1 \circ T^{ah_1}(T^{an}x) f_2 \cdot f_2 \circ T^{bh_1}(T^{bn}x) \prod_{j=1}^3 (g_j \cdot g_j \circ S^{jh_1}) \circ S^{jn} d\nu \right|\\
&= \frac{2}{H_1} + \frac{4}{H_1} \sum_{h_1=1}^{H_1} \left| \int \frac{1}{N} \sum_{n=1}^{N-h_1} f_1 \cdot f_1 \circ T^{ah}(T^{an}x) f_2 \cdot f_2 \circ T^{bh_1}(T^{bn}x) \prod_{j=1}^3 (g_j \cdot g_j \circ S^{jh_1}) \circ S^{(j-1)n} d\nu \right|\\
&\leq \frac{2}{H_1} + \left(\frac{16}{H_1} \sum_{h_1=1}^{H_1} \norm{ \frac{1}{N} \sum_{n=1}^{N-h_1} f_1 \cdot f_1 \circ T^{ah_1}(T^{an}x) f_2 \cdot f_2 \circ T^{bh_1}(T^{bn}x) \prod_{j=2}^3 (g_j \cdot g_j \circ S^{jh_1}) \circ S^{(j-1)n} }_{L^2(\nu)}^2 \right)^{1/2}
\end{align*}
We can now apply the inequality $(\ref{dynamicEstimateTwoTerms})$ and the Cauchy-Schwarz inequality to obtain
\begin{align*}
&\norm{\frac{1}{N} \sum_{n=1}^N f_1(T^{an}x)f_2(T^{bn}x)\prod_{j=1}^3 g_j \circ S^{jn} }_{L^2(\nu)}^2 \\
&\lesssim_{a, b} \frac{1}{H_1} + \left(\frac{1}{H_1}\sum_{h_1 = 1}^{H_1} \left( \frac{1}{H_2} +  \right. \right. \\
&\left. \left. \left(\frac{1}{H_2}\sum_{h_2 = 1}^{H_2} \sup_{t \in \RR} \left| \frac{1}{N} \sum_{n=1}^{N-h_1-h_2} F_{1, h_1} \cdot F_{1, h_1} \circ T^{ah_2}(T^{an}x)F_{2, h_1} \cdot F_{2, h_1} \circ T^{bh_2}(T^{bn}x)e^{2\pi i n t} \right|^2 \right) \right) \right)^{1/4}
\end{align*}
Therefore, we have shown that the averages converge to $0$ in $L^2(\nu)$ when $x \in \tilde{X}_3$.

One of the key observations in showing that the case $k = 2$ implies the case $k = 3$ was the use of the inequality $(\ref{dynamicEstimateTwoTerms})$. We will show that this can be done for $k \geq 4$. For the following lemma, we will use the following notations for our convenience: We shall denote $a_1 = a$ and $a_2 = b$. Let $\vec{h}_l = (h_1, h_2, \ldots, h_l) \in \NN^l$. With this notation, we define the following functions recursively:
\[ \begin{array}{ll}
F_{1, \vec{h}(1)} = f_1 \cdot f_1 \circ T^{a_1h_1},
& F_{2, \vec{h}(1)} = f_2 \cdot f_2 \circ T^{a_2h_1}, \\
F_{1, \vec{h}(2)} = F_{1, \vec{h}(1)} \cdot F_{1, \vec{h}(1)} \circ T^{a_1h_2},
& F_{2, \vec{h}(2)} = F_{2, \vec{h}(1)} \cdot F_{2, \vec{h}(1)} \circ T^{a_2h_2}, \\
\cdots, & \cdots, \\
F_{1, \vec{h}(k-1)} = F_{1, \vec{h}(k-2)} \cdot F_{1, \vec{h}(k-2)} \circ T^{a_1h_{k-1}},
& F_{2, \vec{h}(k-1)} = F_{2, \vec{h}(k-2)} \cdot F_{2, \vec{h}(k-2)}\circ T^{a_2h_{k-1}} \, .
\end{array} \]

\begin{lemma}\label{dynEstMultLemma}
Let everything as in (a) of Theorem $\ref{mainThm}$. Then for each positive integer $k \geq 2$, we have
\begin{align}\label{dynEstMult}
&\limsup_{N \to \infty} \norm{\frac{1}{N} \sum_{n=1}^N f_1(T^{a_1n}x)f_2(T^{a_2n}x) \prod_{i=1}^k g_i \circ S^{in}}_{L^2(\nu)}^2 \\
&\lesssim_{a_1, a_2} \liminf_{H_1 \to \infty}  \left(\frac{1}{H_1} \sum_{h_1=1}^{H_1} \liminf_{H_2 \to \infty} \frac{1}{H_2} \sum_{h_2=1}^{H_2} \cdots \right. \nonumber \\
& \left. \liminf_{H_{k-1} \to \infty} \frac{1}{H_{k-1}} \sum_{h_{k-1}=1}^{H_{k-1}} \limsup_{N \to \infty} \sup_{t \in \RR} \left| \frac{1}{N} \sum_{n=1}^N F_{1, \vec{h}(1)}(T^{a_1n}x)F_{2, \vec{h}(1)}(T^{a_2n}x) e^{2\pi i n t} \right|^2 \right)^{2^{-(k-1)}} \nonumber
\end{align}
\end{lemma}
\begin{proof}
	We will show this by using induction. The base case $k = 2$ has been treated by the estimate $(\ref{dynamicEstimateTwoTerms})$ after we let $N \to \infty$ and $H \to \infty$. Now suppose the estimate holds when we have $k-1$ terms. By applying the van der Corput's lemma and the Cauchy-Schwarz inequality, the left hand side of the estimate $(\ref{dynEstMult})$ is bounded above by the universal constant depending on $a_1$ and $a_2$ times 
	\[ \liminf_{H_1 \to \infty} \left( \frac{1}{H_1} \sum_{h_1 = 1}^{H_1} \limsup_{N \to \infty} \norm{\frac{1}{N} \sum_{n=1}^N F_{1, \vec{h}(k-1)} (T^{a_1n}x)F_{2, \vec{h}(k-1)} (T^{a_2n}x) \prod_{i=2}^k (g_i \cdot g_i \circ S^{ih_1}) \circ S^{(i-1)n}}_{L^2(\nu)}^2 \right)^{1/2} \, , \] 
	and we can apply the inductive hypothesis on this $\limsup$ of the square of the $L^2$-norm and the Cauchy-Schwarz inequality to obtain the desired estimate.
\end{proof}
\begin{proof}[Proof of Theorem $\ref{mainThm}(a)$]
  The set $\tilde{X}_1$ can be obtained from the double recurrence Wiener-Wintner result \cite{WWDR} by applying the spectral theorem. For $k \geq 2$, we consider a set
  \begin{align*}&\tilde{X}_k = \left\{x \in X: \liminf_{H_1 \to \infty} \left(\frac{1}{H_1} \sum_{h_1=1}^{H_1}  \liminf_{H_2 \to \infty} \frac{1}{H_2} \sum_{h_2=1}^{H_2} \cdots \right. \right. \\
  &\left. \left. \liminf_{H_k \to \infty}\frac{1}{H_{k-1}} \sum_{h_k=1}^{H_k} \limsup_{N \to \infty} \sup_{t \in \RR} \left| \frac{1}{N} \sum_{n=1}^N F_{1, \vec{h}(k-1)}(T^{a_1n}x)F_{2, \vec{h}(k-1)}(T^{a_2n}x) e^{2\pi int}\right|^2 \right)^{{2^{-(k-1)}}} = 0 \right\}\, . \end{align*}
  We will show that this set is the desired set of full-measure. To show that $\mu(\tilde{X}_k) = 1$, we will show that the integral 
  \begin{align}
  &\label{inductInt}\int \liminf_{H_1 \to \infty} \left(\frac{1}{H_1}  \sum_{h_1=1}^{H_1}  \liminf_{H_2 \to \infty} \frac{1}{H_2} \sum_{h_2=1}^{H_2} \cdots \right. \\
   &\left. \liminf_{H_{k-1} \to \infty}\frac{1}{H_{k-1}} \sum_{h_{k-1}=1}^{H_{k-1}} \limsup_{N \to \infty} \sup_{t \in \RR} \left| \frac{1}{N} \sum_{n=1}^N F_{1, \vec{h}(k-1)}(T^{a_1n}x)F_{2, \vec{h}(k-1)}(T^{a_2n}x)e^{2\pi int} \right|^2 \right)^{{2^{-(k-1)}}} \, d\mu = 0, \nonumber
   \end{align}
   which would show that the averages inside the integral equals zero for $\mu$-a.e. $x \in X$ since the averages are nonnegative. To do so, we apply Fatou's lemma and H\"{o}lder's inequality to show that the integral above is bounded above by
   \begin{align*}
   & \liminf_{H_1 \to \infty}  \left(\frac{1}{H_1} \sum_{h_1=1}^{H_1} \liminf_{H_2 \to \infty} \frac{1}{H_2} \sum_{h_2=1}^{H_2} \cdots \right. \\
   &\left. \liminf_{H_{k-1} \to \infty}\frac{1}{H_{k-1}} \sum_{h_{k-1}=1}^{H_{k-1}} \int \limsup_{N \to \infty} \sup_{t \in \RR} \left| \frac{1}{N} \sum_{n=1}^N F_{1, \vec{h}(k-1)}(T^{a_1n}x)F_{2, \vec{h}(k-1)}(T^{a_2n}x)e^{2\pi int} \right|^2 \, d\mu \right)^{{2^{-(k-1)}}}.
  \end{align*}
  Note that the integral above is bounded above by $\displaystyle{\min_{i=1, 2} \vertiii{F_{i, \vec{h}(k-1)}}_3^2}$ by the estimate $(\ref{estimate})$. By applying $\liminf_{H_j} \to \infty$ for each $i = 1, 2, \ldots, k-1$, we conclude that the integral on the left hand side of $(\ref{inductInt})$ is bounded above by the minimum of $\vertiii{f_1}_{k+2}^2$ or $\vertiii{f_2}_{k+2}^2$. Since either $f_1$ or $f_2$ belongs to $\mathcal{Z}_{k+1}(T)^\perp$, we know that either $\vertiii{f_1}_{k+2} = 0$ or $\vertiii{f_2}_{k+2} = 0$. Thus, $(\ref{inductInt})$ holds, which implies that $\tilde{X}_k$ is indeed a set of full-measure.
  
  Now we need to show that if $x \in \tilde{X}_k$, then the averages in $(\ref{mainAvg})$ converge to $0$ in $L^2(\nu)$. But this follows immediately from Lemma $\ref{dynEstMultLemma}$, since if $x \in \tilde{X}_k$, the right hand side of $(\ref{dynEstMult})$ is $0$. 
\end{proof}

\section{When both $f_1$ and $f_2$ are in $\mathcal{Z}_{k+1}(T)$ (Proof of (b) and (c) of Theorem $\ref{mainThm}$)}\label{inside}
\subsection{When one of the functions $g_1, g_2, \ldots g_k$ belongs to $\mathcal{Z}_{k}(S)^\perp$}
We first consider the case where either one of the functions $g_1, \ldots, g_k$ belongs to $\mathcal{Z}_{k}(S)^\perp$. In fact, we will show that the averages can be bounded by a seminorm on $L^\infty(\nu)$. We remark here that B. Host and B. Kra have obtained an estimate sharper than the one we provide, using the tools of nilsequences (cf. \cite[Corollary 7.3]{HostKraUniformity}). However, the less sharp estimate that we provide here is sufficient to prove our claim. We will also achieve this estimate without the machinery of nilsequences. 

We prove this for the case $(Y, \mathcal{G}, \nu, S)$ is an ergodic system, and the general case holds by applying an ergodic decomposition on $(Y, S)$. 
\begin{prop}[{See also: \cite[Corollary 7.3]{HostKraUniformity}}]\label{sequence}
	Let $(Y, \mathcal{G}, \nu, S)$ is an ergodic system, $(a_n)_n \in \ell^\infty$ such that $|a_n| \leq 1$ for each $n$, and functions $g_1, \ldots, g_k \in L^\infty(\nu)$. Then
	\begin{equation}\label{seqLim} \limsup_{N \to \infty} \norm{\frac{1}{N} \sum_{n=0}^{N-1} a_n \prod_{i=1}^k g_i \circ S^{in} }_{L^2(\nu)}^2 \leq  2^{k+1} \min_{1 \leq i \leq k} i \cdot \vertiii{g_i}_{k+1}^2. \end{equation}
	\end{prop}
\begin{proof}
	We proceed by induction on $k$. For the case $k = 1$, we apply van der Corput's lemma and the Cauchy-Schwarz inequality to obtain
	\begin{align*}
	&\limsup_{N \to \infty}\norm{\frac{1}{N} \sum_{n=0}^{N-1}a_n g_1 \circ S^n}_{L^2(\nu)}^2\\
	&\leq \liminf_{H \to \infty} \left(\frac{2}{H} + \frac{4}{H} \sum_{h=0}^{H-1} \limsup_{N \to \infty} \left|  \left(\frac{1}{N} \sum_{n=0}^{N-h} a_n\overline{a}_{n+h}\right) 
	\int g_1 \cdot g_1 \circ S^h d\nu \right| \right) \\
	&\leq \liminf_{H \to \infty} \left(\frac{2}{H} + 4 \left(\frac{1}{H}\sum_{h=0}^{H-1}\left| \int g_1 \cdot g_1 \circ S^h d\nu \right|^2 \right)^{1/2}\right) = 4\vertiii{g_1}_2^2,
	\end{align*}
	which proves the base case.
	
	Now suppose the statement holds for $k = l-1$; i.e. we assume that for any $(b_n)_n \in \ell^\infty$ and $G_1, \ldots, G_{l-1} \in L^\infty(\nu)$, we have
	\begin{equation}\label{indHypEstG}
	\limsup_{N \to \infty} \norm{\frac{1}{N} \sum_{n=0}^{N-1} b_n \prod_{i=1}^{l-1} G_i \circ S^{in}}_{L^2(\nu)}^2 \leq 2^{l+1} \min_{1 \leq i \leq l} i \cdot \vertiii{G_i}_{l}^2
	\end{equation}
	To prove this for the case $k = l$, we again apply van der Corput's lemma and the Cauchy-Schwarz inequality to obtain
	\begin{align*}
	&\limsup_{N \to \infty} \norm{\frac{1}{N} \sum_{n=0}^{N-1} a_n \prod_{i=1}^l g_i \circ S^{in}}_{L^2(\nu)}^2 \\
	&\leq \liminf_{H \to \infty} \frac{4}{H} \sum_{h=0}^{H-1} \limsup_{N \to \infty}\left| \int \frac{1}{N} \sum_{n=0}^{N-1} a_n \overline{a}_{n+h} \prod_{i=1}^l \left(g_i \cdot g_i \circ S^{ih}\right) \circ S^{(i-1)n} d\nu \right|\\
	&\leq \liminf_{H \to \infty} \frac{4}{H} \sum_{h=0}^{H-1}  \limsup_{N \to \infty}\norm{\frac{1}{N} \sum_{n=0}^{N-1} a_n \overline{a}_{n+h} \prod_{i=2}^l \left(g_i \cdot g_i \circ S^{ih}\right) \circ S^{(i-1)n}}_{L^2(\nu)} \, .
	\end{align*}
	By setting $b_n = a_n \overline{a}_{n+h}$ and $G_i = g_{i-1} \cdot g_{i-1} \circ S^{(i-1)h}$ for each $h$, we can apply the inequality  $(\ref{indHypEstG})$ to show that
	\begin{align*}
	\limsup_{N \to \infty} \norm{\frac{1}{N} \sum_{n=0}^{N-1} a_n \prod_{i=1}^l g_i \circ S^{in}}_{L^2(\nu)}^2 
	&\leq 4 \min_{2 \leq i \leq l}\liminf_{H \to \infty} \frac{2^{l+1}}{H} \sum_{h=0}^{H-1} \vertiii{g_i \cdot g_i \circ S^{ih}}_l \\
	&\leq 2^{l+3} \min_{2 \leq i \leq l} i \cdot \liminf_{H \to \infty} \left( \frac{1}{H} \sum_{h=0}^H \vertiii{g_i \cdot g_i \circ S^{ih} }_l^{2^l} \right)^{2^{-l}}\\
	&= 2^{l+3} \min_{2 \leq i \leq l} i \cdot \vertiii{g_i}_{l+1}^2 \, .
	\end{align*}
	To keep $\vertiii{g_1}_{l+1}^2$, we compute
	\begin{align*}
		&\limsup_{N \to \infty} \norm{\frac{1}{N} \sum_{n=0}^{N-1} a_n \prod_{i=1}^l g_i \circ S^{in}}_{L^2(\nu)}^2 \\
		&\leq \liminf_{H \to \infty} \frac{4}{H} \sum_{h=0}^{H-1} \limsup_{N \to \infty}\left| \int \frac{1}{N} \sum_{n=0}^{N-1} a_n \overline{a}_{n+h} \prod_{1 \leq i \leq l, i \neq j} \left(g_i \cdot g_i \circ S^{ih}\right) \circ S^{(i-j)n} d\nu \right|\\
		&\leq 2^{2l+3} \min_{1 \leq i \leq l, i \neq j} i \cdot \vertiii{g_i}_{l+1}^2
	\end{align*}
	for a fixed $j$. When these estimates are combined, we have
	\begin{equation}\label{theIneq}\limsup_{N \to \infty} \norm{\frac{1}{N} \sum_{n=0}^{N-1} a_n \prod_{i=1}^l g_i \circ S^{in}}_{L^2(\nu)}^2 \leq 2^{l+3} \min_{1 \leq i \leq l} i \cdot \vertiii{g_i}^2_{l+1} \, , \end{equation}
	which completes the proof.
\end{proof}
With this estimate, Theorem $\ref{mainThm}$(b) can be proven immediately.
\begin{proof}[Proof of (b) of Theorem $\ref{mainThm}$]
Set $a_n = f_1(T^{an}x)f_2(T^{bn}x)$ in Proposition $\ref{sequence}$. Since $f_1, f_2 \in L^\infty$, there exists a set of full-measure $\hat{X}_k$ for which the sequence $a_n \in \ell^\infty$. Because one of the functions $g_1, \ldots, g_k$ belongs to $\mathcal{Z}_k(S)^\perp$, we must have $\vertiii{g_i}_k = 0$ for one of them. Hence, we know from $(\ref{seqLim})$ that the averages must converge to $0$ in $L^2(\nu)$.
\end{proof}
\subsection{When all of the functions $g_1, \ldots, g_k$ belong to $\mathcal{Z}_k(S)$}
Here we use Leibman's pointwise convergence result on nilmanifold to show that the averages converges if all the functions belong to the appropriate Host-Kra-Ziegler factors.
\begin{theorem} [Leibman, \cite{Leibman}]\label{LeibmanThm}
	Let $p(n)$ be a polynomial sequence in a nilpotent Lie group $G$. For any $x \in X = G/\Gamma$ for some discrete compact subgroup $\Gamma$, $F \in \mathcal{C}(X)$, 
	\[\lim_{N \to \infty} \frac{1}{N} \sum_{n=1}^N F(p(n)x)\]
	exists.
\end{theorem}
Using this result, we are now ready to prove the last piece of the main theorem.
\begin{proof}[Proof of (c) of Theorem $\ref{mainThm}$.]	
	With appropriate factors maps, we assume $(X, \mathcal{F}, \mu, T)$ and $(Y, \mathcal{G}, \nu, S)$ to be nilsystems, i.e. $X = G_1/\Gamma_1$, and $Y = G_2/\Gamma_2$, where $G_1$ is a $(k+1)$-step nilpotent Lie group, $G_2$ is a $k$-step nilpotent Lie group , and $\Gamma_1$ and $\Gamma_2$ are discrete co-compact subgroups of $G_1$ and $G_2,$ respectively. In this proof, we will assume that $f_1, f_2 \in \mathcal{C}(X)$, and $g_1, \ldots, g_k \in \mathcal{C}(Y)$. By taking the product of $X^2$ and $Y^k$, we would have another nilmanifold:
	\[ X^2 \times Y^k = (G_1 / \Gamma_1)^2 \times (G_2 / \Gamma_2)^k \cong (G_1^2 \times G_2^k)/(\Gamma_1^2 \times \Gamma_2^k). \]
	
	Let $\tau \in G_1$ such that the action of $\tau$ on an element of $X$ is determined to be $\tau \cdot x = Tx$. Similarly, we define $\sigma \in G_2$ so that $\sigma \cdot y = Sy$. We define a polynomial sequence $p$ on $X^2 \times Y^k$ as follows:
	\[ p(n) = (\tau^{an}, \tau^{bn}, \sigma^n, \sigma^{2n}, \ldots, \sigma^{kn}). \]
	Clearly, $p(n) \in G_1^2 \times G_2^k$ for all $n \in \ZZ$, and it acts on $X^2 \times Y^k$ in a way that
	\[ p(n) \cdot (x_1, x_2, y_1, \ldots, y_k) = (T^{an}x_1, T^{bn}x_2, S^ny_1, \ldots, S^{kn}y_k). \]

	Define a continuous function $F \in \mathcal{C}(X^2 \times Y^k)$ such that
	\[ F(x_1, x_2, y_1, \ldots, y_k) = f_1(x_1)f_2(x_2) \prod_{j=1}^k g_j(y_j).\]
	Theorem $\ref{LeibmanThm}$ tells us that the averages
	\[ \frac{1}{N} \sum_{n=1}^N F(p(n)\cdot(x_1, x_2, y_1, \ldots, y_k)) = \frac{1}{N} \sum_{n=1}^N f_1(T^{an}x_1)f_2(T^{bn}x_2)\prod_{j=1}^k g_j(S^{jn}y_j)\]
	converge for all $(x_1, x_2, y_1, \ldots, y_k) \in X^2 \times Y^k$. So in particular, if the averages were taken a point $(x, x, y, \ldots, y) \in X^2 \times Y^k$ for any $x \in X$ and $y \in Y$, the desired convergence result holds.
	
	By a standard approximation argument, we can extend this result for the case $f_1, f_2 \in L^\infty(\mu) \cap \mathcal{Z}_{k+1}(T)$ and $g_1, \ldots, g_k \in L^\infty(\nu) \cap \mathcal{Z}_{k}(S)$. In this process, we neglect a null-set for which the averages may not converge, which allows us to obtain a set of full-measure $X_k' \subset X$ that satisfies (c) of Theorem $\ref{mainThm}$. 
\end{proof}

\section*{Remarks}
\begin{enumerate}
\item Recently, the first author announced in \cite{AssaniDRnil} that the Wiener-Wintner result obtained in \cite{WWDR} can be extended to a nilsequence Wiener-Wintner result, providing a positive answer to the question raised by B. Weiss in 2014 Ergodic Theory Workshop at UNC Chapel Hill. A similar result was also recently announced by P. Zorin-Kranich \cite{ZK_NilseqWW}.

\item {Recently, we have extended Theorem \ref{mainResult} so that the sequence $a_n = f_1(T^{an}x)f_2(T^{bn}x)$ is $\mu$-a.e. a good universal weight for multiple recurrent averages with commuting transformations. More precisely, we have shown the following:
\begin{theorem*}[{\cite[Theorem 1.1]{CommNorm}}]
	Let $(X, \mathcal{F}, \mu, T)$ be a measure-preserving system, and suppose $f_1, f_2 \in L^\infty(\mu)$. Then there exists a set of full-measure $X_{f_1, f_2}$ such that for any $x \in X_{f_1, f_2}$, for any $a, b \in \ZZ$ and any positive integer $k \geq 1$, for any other measure-preserving system with $k$ commuting transformations $(Y, \mathcal{G}, \nu, S_1, S_2, \ldots S_k)$, and for any $g_1, g_2, \ldots g_k \in L^\infty(\nu)$, the averages
	\begin{equation*} \frac{1}{N} \sum_{n=0}^{N-1} f_1(T^{an}x)f_2(T^{bn}x)\prod_{i=1}^k g_i \circ S_i^n \text{ converge in } L^2(\nu).\end{equation*}
\end{theorem*}
In other words, in terms of Definition $\ref{guweight}$, we have shown that for $\mu$-a.e. $x \in X$, the sequence $(f_1(T^{an}x)f_2(T^{bn}x))_n$ is a good universal weight for the process $(X_n)_n$ for norm convergence, where $(X_n)_n$ is of the form
\[ X_n = \prod_{i=1}^k g_i \circ S_i^n, \]
where $g_1, g_2, \ldots, g_k \in L^\infty(\nu)$ for any measure-preserving system with commuting transformations $(Y, \mathcal{G}, \nu, S_1, S_2, \ldots, S_k)$, for any positive integer $k \geq 1$.
}
\end{enumerate}

\bibliographystyle{plain}
\bibliography{RM_Bib}

\end{document}